\documentclass[twoside,11pt]{article}
\RequirePackage[OT1]{fontenc}
\RequirePackage{amsmath}

\usepackage{amscd,amsfonts,amsmath,amstext,enumerate,float,amssymb,epsfig,multirow,color,latexsym,mathrsfs,tocloft,tikz,tkz-fct,listofsymbols,tabu,stmaryrd}
\usepackage{amsthm,esint}
\usepackage[numbers]{natbib}
\usepackage{courier}
\usepackage{bbm}
\usepackage{booktabs}
\usepackage{subcaption}
\usepackage{graphicx}
\usepackage[titletoc]{appendix}
\usetikzlibrary{decorations.pathreplacing}
\usepackage[letterpaper,left=3.81cm,top=2.54cm,bottom=2.54cm,right=2.54cm]{geometry}
\usepackage[nottoc]{tocbibind}

\usepackage{xparse}

\usepackage{appendix}
\usepackage{chngcntr}
\usepackage{etoolbox}

\usepackage{tikz,tikz-cd}
\usepackage{tkz-euclide}
\usetkzobj{all}
\usetikzlibrary{matrix,shapes,arrows}
\usetikzlibrary{positioning}
\usepackage{smartdiagram}

\usepackage{graphicx}
\usepackage{sidecap}
\usepackage{float}
\usepackage{enumerate,mdwlist}
\usepackage{epsfig}
\usepackage[bbgreekl]{mathbbol}
\usepackage{amsfonts,amsmath,enumerate}
\usepackage{amscd}
\usepackage{amsthm}
\usepackage{mleftright}
\DeclareSymbolFontAlphabet{\mathbb}{AMSb}
\DeclareSymbolFontAlphabet{\mathbbl}{bbold}
\usepackage{colortbl}

\usepackage{amsmath,amsfonts,amssymb,mathrsfs}\usepackage{bm}
\usepackage{latexsym}
\usepackage{mleftright}
\usepackage{amssymb}
\usepackage{color}
\usepackage{xifthen}

\usepackage[titletoc]{appendix}

\usepackage{footnote}

\newlength{\oldparindent}
\usepackage{hyperref} 
\hypersetup{
	colorlinks = true,
	linkcolor = blue,
	anchorcolor = blue,
	citecolor = blue,
	filecolor = blue,
	urlcolor = blue
}

\usepackage[ruled,vlined,resetcount,algosection]{algorithm2e}
\usepackage{makeidx}
\usepackage{hyperref}
\usepackage{amsmath}
\usepackage{amsthm}
\usepackage{amsfonts}
\usepackage{amssymb}
\usepackage{mathrsfs}
\usepackage{graphicx}
\usepackage{tikz,tikz-cd}
\usepackage{tkz-euclide}
\usetkzobj{all}
\usetikzlibrary{matrix}
\usepackage{smartdiagram}

\NewDocumentCommand{\ffff}{oo}{{
\mathscr{F}\IfValueT{#1}{_{{#1}}}\IfValueT{#2}{\left(
	{#2}
	\right)}\IfValueF{#2}{\left(X,d\right)}
}}

\makeatletter
\newcommand{\superimpose}[2]{%
	{\ooalign{$#1\@firstoftwo#2$\cr\hfil$#1\@secondoftwo#2$\hfil\cr}}}
\makeatother

\newcommand{\rr}{{\mathbb{R}}}
\newcommand{\rrflex}[1]{{\ensuremath{\rr^{#1}
}}}

\newcommand{\rrd}{{\rrflex{d}}}

\newcommand{\cc}{{\mathbb{C}}}

\newcommand{\zz}{{\mathbb{Z}}}

\newcommand{\thetareg}[1]{{\ensuremath{
			\theta_t^{\mathfrak{ext}}
}}}
\newcommand{\thetarreg}[1]{{\ensuremath{
			\thetar{t}%
}}}

\newcommand{\thetar}[1]{{
		\ensuremath{
			\theta_{#1}^{\mathfrak{ext}}
		}
}}

\makeatletter
\newcommand*\bigcdot{\mathpalette\bigcdot@{.5}}
\newcommand*\bigcdot@[2]{\mathbin{\vcenter{\hbox{\scalebox{#2}{$\m@th#1\bullet$}}}}}
\makeatother

\makeatletter
\def\@chapter[#1]#2{\ifnum \c@secnumdepth >\m@ne
	\refstepcounter{chapter}%
	\typeout{\@chapapp\space\thechapter.}%
	\addcontentsline{toc}{chapter}%
	{\protect\numberline{\thechapter}\string\hypertarget{chap\thechapter}{#1}}%
	\else
	\addcontentsline{toc}{chapter}{#1}%
	\fi
	\chaptermark{#1}%
	\addtocontents{lof}{\protect\addvspace{10\p@}}%
	\addtocontents{lot}{\protect\addvspace{10\p@}}%
	\if@twocolumn
	\@topnewpage[\@makechapterhead{#2}]%
	\else
	\@makechapterhead{#2}%
	\@afterheading
	\fi}
\def\@makechapterhead#1{%
	\vspace*{50\p@}%
	{\parindent \z@ \raggedright \normalfont
		\ifnum \c@secnumdepth >\m@ne
		\huge\bfseries \@chapapp\space \thechapter
		\par\nobreak
		\vskip 20\p@
		\fi
		\interlinepenalty\@M
		\Huge \bfseries \hyperlink{chap\thechapter}{#1}\par\nobreak
		\vskip 40\p@
}}
\makeatother

\NewDocumentCommand\SpageGen{oo}{
	{
		\left[
		\IfValueT{#1}{{#1}}\IfValueF{#1}{\xx}
		,
		\IfValueT{#2}{{#2}}\IfValueF{#2}{f}
		\right]
	}
}
\newcommand{\xx}{\bbxi}

\newcommand{\xireg}[1]{{\ensuremath{
			\xi_t^{\mathfrak{r}}
}}}
\newcommand{\xirreg}[1]{{\ensuremath{
			\xir{t}%
}}}

\newcommand{\xir}[1]{{
		\ensuremath{
			\xi_{#1}^{\mathfrak{r}}
		}
}}

\NewDocumentCommand\llangle{mo}{
	\left\langle\IfValueT{#1}{{#1}}
	\IfValueT{#2}{
		,{#2}
	}
	\right\rangle
	\IfValueF{#2}{^+}
}
\NewDocumentCommand\lrangle{m}{
	\left\langle\IfValueT{#1}{{#1}}
	\right\rangle
}

\NewDocumentCommand\NN{oo}{
	{\mathcal{NN}%
		_{\IfValueT{#2}{,#2}}	
		\IfValueF{#1}{^{\ffff}}\IfValueT{#1}{^{#1}}
	}
}

\NewDocumentCommand\NNshal{oo}{
	{nn%
		_{\IfValueT{#1}{#1}}	
		\IfValueF{#2}{^{\ffff}}\IfValueT{#2}{^{#2}}
	}
}
\NewDocumentCommand\NNho{oo}{
	{\mathcal{HNN}%
		_{R\IfValueT{#1}{#1}}	
		\IfValueF{#2}{^{\ffff}}\IfValueT{#2}{^{#2}}
	}
}

\NewDocumentCommand\NNaff{oo}{
	{\mathcal{NN}%
		_{R,\IfValueT{#1}{#1}}	
		\IfValueF{#2}{^{a}}
	}
}

\newcommand{\gggg}{{\mathscr{G}}}	

\NewDocumentCommand\Lpc{o}{
	\Lambda^p_{\mu}\left(\gggg;
	\IfValueT{#1}{{#1}}\IfValueF{#1}{X}
	\right)}
\NewDocumentCommand\LpcK{o}{\Lambda^p_{\mu}\left(K;X\right)}
\NewDocumentCommand\Bor{o}{{
		\mathscr{B}\left(
		\IfValueF{#1}{X}
		\IfValueT{#1}{{#1}}
		\right)
}}

\usepackage{xparse}
\NewDocumentCommand\grp{oo}{
	{
		\mathscr{G}^{(
			\IfValueT{#1}{{#1}}\IfValueF{#1}{{p}}
			)}\left(
		\IfValueF{#2}{\rrd}\IfValueT{#2}{\rrflex{#2}}
		\right)
	}
}

\newtheoremstyle{dotless}{}{}{\itshape}{}{\bfseries}{}{ }{}
\theoremstyle{dotless}

\numberwithin{equation}{section}
\theoremstyle{plain}

\theoremstyle{definition}
\newtheorem{defn}{Definition}[section]
\theoremstyle{definition}

\theoremstyle{plain}
\theoremstyle{definition}
\newtheorem{prop}[defn]{Proposition}%
\newtheorem{cor}[defn]{Corollary}%
\newtheorem{lem}[defn]{Lemma}%
\newtheorem{ex}[defn]{Example}%
\newtheorem{thrm}[defn]{Theorem}%
\newtheorem{rremark}[defn]{Remark}%
\newcommand{\nf}{ \Omega^n(A/k) }

\newcommand{\ncof}{ \Omega^1(A/k) }

\newcommand{\kiki}{{k_{i^{-1}[\ma]}}}

\newcommand{\ma}{ {\mathfrak{m}} }

\newcommand{\ama}{ A_{\ma} }

\newcommand{\adj}{\rotatebox[origin=c]{180}{\ensuremath\vdash}}

\newcommand{\rh}{ \mathscr{E}_{A^e}^k }

\newcommand{\ezz}{\mathscr{E}_{\zz}^{\zz}}

\newcommand{\ea}{\mathscr{E}_A^k}

\newcommand{\eal}{\mathscr{E}_{\mathfrak{m},k}}

\newcommand{\eall}{\mathscr{E}_{A_{\ma}}^{k_{i^{-1}[\ma]}}}

\newcommand{\ek}{\mathscr{E}^k}

\newcommand{\abr}{\hat{CB}_{\star}(A)}

\newcommand{\an}{A^{\otimes n}}

\newcommand{\antw}{A^{\otimes n+2}}

\usepackage{fancyhdr}
\begin{document}
%%%	
\author{Anastasis Kratsios
\thanks{Department of Mathematics, ETH Z\"{u}rich, HG G 32.3, R\"{a}mistrasse 101, 8092 Z\"{u}rich.  email: anastasis.kratsios@math.ethz.ch}
}
\title{\Large \bf A Lower-Bound on the Hochschild Cohomological Dimension} %
\lhead{A. Kratsios
}
\chead{\small }
\rhead{\small \today}%

\date{%
	\today %
}
	
	\maketitle

	\begin{abstract}
A concrete lower-bound for the Hochschild cohomological dimension of a commutative $k$-algebra, in terms of three other homological invariants is obtained.  This result is then used to show that most $k$-algebras fail to be quasi-free, even if they are smooth.  This result generalizes a result of \cite{cuntz1995algebra} to the case where the base-ring is no longer $\cc$ but can be any commutative ring with unity.  
	\end{abstract}

	
	\noindent
	\let\thefootnote\relax\footnotetext{This research was supported by the ETH Z\"{u}rich Foundation.}
	
	
	%%% HERE!!!
	
	\section*{Introduction}
    Non-commutative algebraic geometry is a rapidly developing area of contemporary mathematical research which studies spaces which are dual to certain non-commutative algebras.  Many natural extensions and analogues of classical commutative algebraic geometry notions have studied since the field's inception, from differential forms in the celebrated results of \cite{HKR}, to the smoothness of a non-commutative space as studied by \cite{krahmer2006poincare}.  

This paper will focus on the later of these topics.  This paper examines how the quasi-freeness property of \cite{cuntz1995algebra}, which is a standard extension of the formal smoothness introduced in \cite{grothendieck2002rev}, differs from classical formal smoothness when the underlying algebra corresponds to a smooth affine algebraic variety (in the scheme-theoretic sense) over a commutative ring $k$.  Specifically, using the central tool of non-commutative algebraic geometry, the Hochschild cohomology, a homological dimension theoretic argument, is used to show that most smooth affine algebras fail to be smooth in the non-commutative sense; that is they fail to be quasi-free.  Examples over $\zz$ are also explored.

This paper is organized as follows.  Section $1$ briefly overviews some of the relevant language and tools from relative homological algebra, homological dimension theory, and non-commutative algebraic geometry.  Section $2$ then makes use of these tools to establish a lower bound on the Hochschild cohomological dimension of a commutative algebra.  The result is then used to examine the disparity between quasi-freeness and classical formal smoothness of the algebra of function corresponding to a smooth affine algebraic variety.  The appendix contains a small number of auxiliary results, whose proof would otherwise detract from the central flow of this article.  
\section{Background}
The results in this paper are formulated using the \textit{relative homological algebra} introduced in \cite{beck1967triples}.  The theory is analogous to classical homological algebra; see \cite{rotman1979introduction}, for example, but in this case, one builds the entire theory relative to a suitable subclass of epi(resp. mono)-morphisms.  In our case, these are defined as follows.  

\begin{defn}[$\mathscr{E}_A^k$-Epimorphism]\label{defn_rel_epi}
	For any $k$-algebra $A$, an epimorphism $\epsilon$ in $_{A}Mod$ is an {$\mathscr{E}_A^k$-epimorphism} \textit{if and only if} $\epsilon$'s underlying morphism of $k$-modules is a $k$-split epimorphism in $_kMod$.  
	The class of these epimorphisms is denoted $\mathscr{E}_A^k$. 
\end{defn}
%
%\begin{rremark} $\label{remepinotsomuchchch1}$
%    Straightaway from this definition it follows that the class of \textit{all epimorphisms in $_AMod$} always contains $\mathscr{E}_A^k$ as a subclass \textit{(though the containment is not necessarily proper)}.  
%\end{rremark}
\begin{defn}[$\ea$-Exact sequence]\label{defn_rel_exact}
	An exact sequence of $A$-modules: 
	\begin{equation}
	\begin{tikzpicture}[>=angle 90]
	\matrix(a)[matrix of math nodes,
	row sep=3em, column sep=1.5em,
	text height=1.5ex, text depth=0.25ex]
	{..& M_i & M_{i+1} & M_{i+2} &.. \\};
	\path[->,font=\scriptsize]
	(a-1-1) edge node[above]{$ \phi_{i-1} $} (a-1-2);
	\path[->,font=\scriptsize]
	(a-1-2) edge node[above]{$ \phi_{i} $}(a-1-3);
	\path[->,font=\scriptsize]
	(a-1-3) edge node[above]{$ \phi_{i+1} $}(a-1-4);
	\path[->,font=\scriptsize]
	(a-1-4) edge node[above]{$ \phi_{i+2} $}(a-1-5);
	\end{tikzpicture}
	\label{exactseqAmod1}
	\end{equation}
	is said to be \textbf{$\ea$-exact} if and only if for every integer $i$ the there exists a morphism of $k$-modules $\psi_i: M_{i+1}\rightarrow M_i $ such that: 
	\begin{equation}
	\phi_i = \phi_i \circ \psi_i \circ \phi_i .
	\label{reladmissibility}
	\footnote{Property~\eqref{reladmissibility} is called $\rh$-admissibility \cite{hilton1971course} (alternatively it is called $\rh$-allowable \cite{mac1995c}).  }
	\end{equation}
	In particular, a short exact sequence of $A$-modules which is $\ea$-exact is called an {$\ea$-short exact sequence}.  
\end{defn}
\begin{ex}\label{exadmis}
	The augmented bar complex $\abr$ of a $k$-algebra $A$ is $\rh$-exact.  
\end{ex}
%\begin{proof}
%    For every $n \in \mathbb{N}$ let $s_{n}: CB_n(A) \rightarrow CB_{n+1}(A)$ be as in~\eqref{r44v4v44444444}.  $s_n $ is $k$-linear since: \small \begin{equation}
%    \mbox{Let }
%    a_0\otimes_k ... \otimes_k a_{n+1} , a_0'\otimes_k ... \otimes_k a_{n+1}' \in CB_n(A)\mbox{, }c \in k 
%    \end{equation}
%    \begin{equation}
%    s_n(a_0\otimes_k ... \otimes_k a_{n+1} +c a_0'\otimes_k ... \otimes_k a_{n+1}')
%    \end{equation}
%    \begin{equation}
%    =1 \otimes_k a_0\otimes_k ... \otimes_k a_{n+1} +c1 \otimes_k  a_0'\otimes_k ... \otimes_k a_{n+1}') 
%    \end{equation}
%    \begin{equation}
%    = s_n(a_0\otimes_k ... \otimes_k a_{n+1}) +c s_n(  a_0'\otimes_k ... \otimes_k a_{n+1}').  
%    \end{equation}
%    \normalsize
%    
%    The $s_n$ show that each $b_n'$ satisfies property~\eqref{reladmissibility} since:  
%    \begin{equation}
%    b_n'  \circ s_{n-1}\circ b_n'
%    \end{equation}
%    \begin{equation}
%    =b_n'\circ(1 + b_{n+1}' \circ s_n)
%    \end{equation}
%    \begin{equation}
%    =b_n'  =b_n'\circ b_{n+1}' \circ s_n
%    \label{bennono1}
%    \end{equation}
%    Since $b_{\star}'$ is a coundary map $b_n\circ b_{n+1} =0$; hence~\eqref{bennono1} equates to: \begin{equation}
%    =  b_n'+0=b_n'.  
%    \end{equation}
%\end{proof}
%
\begin{defn}[$\ea$-Projective module]
	If $A$ is a $k$-algebra and $P$ is an $A$-module, then $P$ is said to be {$\ea$-projective}    \footnote{This definition is equivalent to requiring that $P$ verify the universal property of projective modules \textit{only} on $\ea$-epimorphisms \cite{mac1995c}.  } \textit{if and only if} for every $\ea$-short exact sequence: \begin{equation}
	\begin{tikzpicture}[>=angle 90]
	\matrix(a)[matrix of math nodes,
	row sep=3em, column sep=1.5em,
	text height=1.5ex, text depth=0.25ex]
	{0 & M & N & N' & 0 \\};
	\path[->,font=\scriptsize]
	(a-1-1) edge node[above]{$  $} (a-1-2);
	\path[->,font=\scriptsize]
	(a-1-2) edge node[above]{$ \eta $}(a-1-3);
	\path[->,font=\scriptsize]
	(a-1-3) edge node[above]{$ \epsilon $}(a-1-4);
	\path[->,font=\scriptsize]
	(a-1-4) edge node[above]{$  $}(a-1-5);
	\end{tikzpicture}
	\label{exactseqAmod2}
	\end{equation}
	the sequence of $k$-modules: 
	\begin{equation}
	\begin{tikzpicture}[>=angle 90]
	\matrix(a)[matrix of math nodes,
	row sep=3em, column sep=1.5em,
	text height=1.5ex, text depth=0.25ex]
	{0 & Hom_A(P,M) & Hom_A(P,N) & Hom_A(P,N') & 0 \\};
	\path[->,font=\scriptsize]
	(a-1-1) edge node[above]{$  $} (a-1-2);
	\path[->,font=\scriptsize]
	(a-1-2) edge node[above]{$ \eta^{\star} $}(a-1-3);
	\path[->,font=\scriptsize]
	(a-1-3) edge node[above]{$ \epsilon^{\star} $}(a-1-4);
	\path[->,font=\scriptsize]
	(a-1-4) edge node[above]{$  $}(a-1-5);
	\end{tikzpicture}
	\label{exactseqAmod3}
	\end{equation}
	is exact.  
\end{defn}
\begin{ex}\label{propan}
	$A^{\otimes n+2}$ is $\rh$-projective for all $n\in \mathbb{N}$.
\end{ex}
$\ea$-projective $A$-modules have analogous properties to projective $A$-modules.  For example, $\ea$-projective $A$-modules admit the following characterization.  
\begin{prop}$\label{eaprojchar}$
	For any $A$-module $P$ the following are equivalent: 
	\begin{enumerate}[(i)]
		\item \textbf{$\ea$-Short exact sequence preservation property} $P$ is $\ea$-projective.  
		\item \textbf{$\ea$-lifting property} For every $\ea$-epimorphism $f: N \rightarrow M$ if there exists an $A$-module morphism $g: P \rightarrow M$ then \textit{there exists} an $A$-module map $\tilde{f}: P \rightarrow N$ such that $f \circ \tilde{f} = g$.  
		\item \textbf{$\ea$-splitting property} Every short $\ea$-exact sequence of the form: \begin{equation}
		\begin{tikzpicture}[>=angle 90]
		\matrix(a)[matrix of math nodes,
		row sep=3em, column sep=2.5em,
		text height=1.5ex, text depth=0.25ex]
		{\mathfrak{E}_{\pi}: 0 & M & N & P & 0 \\};
		\path[->,font=\scriptsize]
		(a-1-1) edge node[above]{$  $} (a-1-2);
		\path[->,font=\scriptsize]
		(a-1-2) edge node[above]{$  $} (a-1-3);
		\path[->,font=\scriptsize]
		(a-1-3) edge node[above]{$  $} (a-1-4);
		\path[->,font=\scriptsize]
		(a-1-4) edge node[above]{$  $} (a-1-5);
		\end{tikzpicture} 
		\label{4b4b433}
		\end{equation}
		is $A$-split-exact.  
		\item \textbf{$\ea$-free direct summand property}
		\footnote{If $F$ is a free $k$-module, some authors call $A \otimes_k F$ an $\ea$-free module.  
			In fact this gives an alternative proof that $A^e\otimes_k \an \cong \antw$ is $\rh$-free for every $n \in \mathbb{N}$.  
		}
		There exists a $k$-module $F$, an $A$-module $Q$ and an isomorphism of $A$-modules $\phi: P\oplus Q \overset{\cong}{\rightarrow} A\otimes_k F$.  
	\end{enumerate}
\end{prop}
\begin{proof}
	See \cite{mac1995c} pages 261 for the equivalence of $1$,$2$ and $3$ and page 277 for the equivalence of $1$ and $4$.  
\end{proof}
For a homological algebraic theory to be possible, one needs enough projective (resp. injective) objects.  The next result shows that there are indeed enough $\ea$-projectives in $_AMod$.  
\begin{prop}[Enough $\ea$-projectives]\label{prop_enoughprojrels}
	If $A$ is a $k$-algebra and $M$ is an $A$-module then there exists an $\ea$-epimorphism $\epsilon: P \rightarrow M$ where $P$ is an $\ea$-projective.  
\end{prop}
\begin{proof}
	By proposition~\ref{eaprojchar} $A \otimes_k M$ is $\ea$-projective.  Moreover the $A$-map $\zeta: A \otimes_k M \rightarrow M$ described on elementary tensors as $(\forall a \otimes_k m \in A \otimes_k M) \zeta (a \otimes_k m):= a\cdot m$ is epi and is $k$-split by the section $m \mapsto 1 \otimes_k m$.  
\end{proof}
Since there are enough projective objects, then one can build a resolution of any $A$-module by $\mathscr{E}_A^k$-projective modules. 
\begin{defn}[$\mathscr{E}_A^k$-projective resolution]
	If $M$ is an $A^e$-module then a resolution $P_{\star}$ of $M$ is called an {$\mathscr{E}_A^k$-projective resolution} of $M$ \textit{if and only if} each $P_i$ is an $\mathscr{E}_A^k$-projective module and $P_{\star}$ is an $\ea$-exact sequence.  
\end{defn}
\begin{ex} $\label{leembares1}$
	The augmented bar complex $\abr$ of $A$ is an $\rh$-projective resolution of $A$.  
\end{ex}
\begin{rremark}\label{leembares2}
	A nearly completely analogous argument to example~\ref{leembares1} shows that for any \\
	$(A,A)$-bimodule $M$, $M \otimes_{A} \abr$ is an $\rh$-projective resolution of $M$, see for details \cite{weibel1995introduction}.  
\end{rremark}
Following \cite{hilton1971course}, the $\mathscr{E}_A^k$-relative derived functors of the tensor product and the $Hom_A$-functors are introduced, as follows.  
\begin{defn}
	\textbf{$\mathscr{E}_A^k$-relative Tor}
	
	If $N$ is a right $A$-module, $M$ is an $A$-module and $P_{\star}$ is an $\ea$-projective resolution of $N$ then the $k$-modules
	$H_{\star}(P_{\star}$ $\otimes_{A} M)$ 
	are called the $\mathscr{E}_A^k$-relative Tor $k$-modules of $N$ with coefficients in the $A$-module $M$ and are denoted by $Tor^n_{\mathscr{E}_A^k}(N,M)$.  
\end{defn}
Let $H_{\star}$ (resp. $H^{\star}$) denote the (co)homology functor from the category of chain (co)complexes on an $A$-module to the category of $A$-modules.  The $\mathscr{E}_A^k$-relative Tor functors are defined as follows.  
\begin{ex}\label{tornonosamereltor1}
	The $\mathscr{E}_A^k$-relative Tor functors may differ from the usual \textit{(or "absolute")} Tor functors.  
	For example consider all the ${\mathbb{Z}}$-algebra ${\mathbb{Z}}$, any $\mathbb{Z}$-modules $N$ and $M$ are $\mathscr{E}_{\mathbb{Z}}^{\mathbb{Z}}$-projective.  In particular, this is true for the $\mathbb{Z}$-modules $\mathbb{Z}$ and $\mathbb{Z}/2\mathbb{Z}$.  Therefore $Tor_{\mathscr{E}_{\mathbb{Z}}^{\mathbb{Z}}}^n(\mathbb{Z},\mathbb{Z}/2\mathbb{Z})$ vanish for every positive $n$, however $Tor^n_{\mathbb{Z}}(\mathbb{Z}, \mathbb{Z}/2\mathbb{Z})$ does not.  For example, $Tor^1_{\mathbb{Z}}(\mathbb{Z}, \mathbb{Z}/2\mathbb{Z})\cong \mathbb{Z}/2\mathbb{Z}$ \cite{rotman1979introduction}.  
\end{ex}
Similarly there are $\ea$-relative Ext functors.  
\begin{defn}[$\mathscr{E}_A^k$-relative Ext]
	If $N$ is and $M$ are $A$-modules and $P_{\star}$ is an $\ea$-projective resolution of $N$ then the $k$-modules
	$H^{\star}(Hom_A(P_{\star}, M))$ 
	are called the $\mathscr{E}_A^k$-relative Ext $k$-modules of N with coefficients in the $A$-module $M$ and are denoted by $Ext^n_{\mathscr{E}_A^k}(N,M)$.  
\end{defn}
The $\mathscr{E}_A^k$-relative homological algebra is indeed well defined, since both the definitions of $\ea$-relative Ext and $\ea$-relative Tor are independent of the choice of $\ea$-projective resolution.  
\begin{thrm}[{$\ea$-Comparison theorem}]\label{thmcomparisonrelcohomolg} If $P_{\star}$ and $P'_{\star}$ are $\ea$-projective resolutions of an $A$-module $N$ then for any $A$-module $M$ there are natural isomorphisms: 
	\begin{equation}
	H^{\star}(Hom_{\ea}(P_{\star},N)) \overset{\cong}{\rightarrow}H^{\star}(Hom_{\ea}(P'_{\star},N))
	\end{equation}
	and 
	if $P_{\star}$ and $P'_{\star}$ are $\ea$-projective resolutions of a right $A$-module $N$ then:
	\begin{equation}
	H_{\star}(P_{\star}{ }\otimes_{A}N) \overset{\cong}{\rightarrow} H_{\star}(P'_{\star}{ }\otimes_{A}N)
	\end{equation}
\end{thrm}
\begin{proof}
	Nearly identical to the usual comparison theorem, see \cite{mac1995c}.  
\end{proof}
\begin{ex}
	The $Ext_{\zz}$ and $\ezz$-relative Ext may differ.  For example, one easily computes $Ext_{\zz}^1(\zz,\zz/2\zz) \cong \zz/2\zz$.  However, $Ext_{\ezz}^1(\zz,\zz/2\zz) \cong 0$.  
\end{ex}
%\begin{proof}
%    Since~\eqref{coutnerexseq12} is a $\ezz$-projective resolution of the $\zz$-module $\zz/2\zz$, there are natural isomorphisms of $\zz$-modules: \begin{equation}
%    Ext_{\ezz}^1(\zz,\zz/2\zz) \cong \zz/2\zz.  
%    \end{equation}
%    
%    In contrast, since~\eqref{coutnerexseq12} is a $\ezz$-projective resolution of the $\zz$-module $\zz/2\zz$ then theorem~\ref{thmcomparisonrelcohomolg} implies: \begin{equation}
%    Ext_{\ezz}^1(\zz,\zz/2\zz) \cong 0/0 \cong 0.  
%    \end{equation}
%\end{proof}
%
%
Analogous to the fact that for any $A$-module $P$, $P$ is projective if and only if $Ext_A^1(P,N)\cong 0$ for every $A$-module $N$ there is the following result, which can be found in \citep[Chapter IX]{hilton1971course}.  
\begin{prop}\label{projcharext}
	$P$ is an $\ea$-projective module if and only if for every $A$-module $N$: \begin{equation}
	Ext^1_{\ea}(P,N)\cong 0
	\end{equation}
\end{prop}
\subsection{Hochschild (Co)homological Dimension}
Since $CB_{\star}(A)$ is an $\rh$-projective resolution of $A$ then theorem~\ref{thmcomparisonrelcohomolg} and the definition of the $Ext_{\rh}^{\star}(A,-)$ functors imply that the Hochschild cohomology of $A$ with coefficients in of \cite{hochschild1945onhcomology}, denoted by $HH^{\star}(A,N)$, can be expressed using the $Ext_{\rh}^{\star}$.  We maintain this perspective throughout this entire article.  
\begin{prop}\label{hochrelderived}
	For every $A^e$ module $N$ there are $k$-module isomorphisms, natural in $N$: \begin{equation}
	HH^{\star}(A,N) \overset{\cong}{\rightarrow} Ext_{\rh}^{\star}(A,N)
	\end{equation}
	Taking short $\rh$-exact sequences to isomorphic long exact sequences.  
\end{prop}

\begin{defn}[Hochschild Homology]
	%    \footnote{If $A$ is a commutative $k$-algebra of essentially-finite type and $k$ is Noetherian then $HH_{\star}(A,A) \cong \Omega^n_{A|k}$, where $ \Omega^n_{A|k}$ are the K\"{a}hler $n$-forms \cite{weibel1995introduction}, therefore the Hochschild homology provides yet another non-commutative analogue of $\Omega^n(A|k)$.  }
	%    \footnote{There is a duality relationship between the Hochschild cohomology and the Hochschild Homology modules of a $\cc$-algebra explored in \cite{van1998relation}.  In the case where $A$ is the coordinate ring of a smooth affine algebraic $\mathbb{C}$-variety this relationship becomes even clearer \cite{krahmer2006poincare}.}  
	%    
	The Hochschild homology $HH_{\star}(A,N)$ of a $k$-algebra $A$ with coefficient in the $(A,A)$-bimodule $N$ is defined as: \begin{equation}
	HH_{\star}(A,N) := H_{\star}(P_{\star}{ }\otimes_{A}N)
	\end{equation}
	where $P_{\star}$ is an $\rh$-projective resolution of $A$.  
\end{defn}
Following the results of~\cite{HKR}, the Hochschild cohomology has become the central tool for obtaining non-commutative algebraic geometric analogues of classical commutative algebraic geometric notions.  The one of central focus in this paper, is the Hochschild cohomological dimension,  
\begin{defn}[Hochschild cohomological dimension]\label{defHomdim}
	The {Hochschild cohomological dimension} of a $k$-algebra $A$ is defined as: \begin{equation}
	HCdim(A|k):= \underset{M\in _{A^e}Mod}{sup} ( sup \{n \in \mathbb{N}^{\#} | HH^{n}(A,M) \not{\cong} 0 \} ).  
	\end{equation}  
	Where $\mathbb{N}^{\#}$ is the ordered set of extended natural numbers.  
\end{defn}
The Hochschild cohomological dimension may be related to the following cohomological dimension.
%\begin{defn}[Cuntz-Quillen n-Forms]\label{propdesccqforms}    
%    For any natural number $n$ and any $k$-algebra $A$ the module of $n$-Cuntz-Quillen forms on $A$ is defined as:
%    \begin{equation}
%    \cq := Ker(\bar{b}'_{n-1}: \bar{CB}_{n} \rightarrow \bar{CB}_{n-1})
%    \end{equation}
%\end{defn}
%\begin{defn}[{$\nf$}]
%    Let $A$ be a $k$-algebra and $n \in \mathbb{N}$, then define: \begin{equation}
%    \nf := Ker(b_{n-1}')
%    \end{equation}
%    where $b_{n-1}'$ is the $(n-1)^{th}$ differential in the augmented Bar resolution of $A$; see \cite{weibel1995introduction} for details.  
%\end{defn}
\begin{defn}[$\ea$-projective dimension]
	If $n$ is an natural number and $M$ is an $A$-module then $M$ is said to be of {$\ea$-projective dimension} at most $n$ if and only if there exists a deleted $\ea$-projective resolution of $M$ of length $n$.  If no such $\ea$-projective resolution of $M$ exists then $M$ is said to be of $\ea$-projective dimension $\infty$.  The $\ea$-projective dimension of $M$ is denoted \textbf{$pd_{\ea}(M)$}.  
\end{defn}
The following is a translation of a classical homological algebraic result into the setting of $\rh$-projective dimension, $\nf$ and Hochschild cohomology.  Here, $\nf\triangleq Ker(b_{n-1}')$ and $b_{n-1}'$ is the $(n-1)^{th}$ differential in the augmented Bar resolution of $A$; see \cite{weibel1995introduction} for details on the augmenter Bar complex.  
\begin{thrm}\label{thrmcharnsmooth}
	For every natural number $n$, the following are equivalent:
	\begin{enumerate}
		\item $HCdim(A|k)\leq n$
		\item $A$ is of $\rh$-projective dimension at most $n$
		\item $\nf$ is an $\rh$-projective module.  
		\item $HH^{n+1}(A,M)$ vanishes for every $(A,A)$-bimodule $M$.  
		\item $Ext_{\rh}^{n+1}(A,M)$ vanishes for every $A^e$-module $M$.  
	\end{enumerate}
\end{thrm}
\begin{proof} (1 $\Rightarrow$ 4)    By definition of the Hochschild cohomological dimension.  (4 $\Leftrightarrow$ 5)  By proposition~\ref{hochrelderived}.  (3 $\Rightarrow$ 2)  Since $\nf$ is $\rh$-projective: 
	$$
	0 \rightarrow \nf \rightarrow CB_{n-1}(A) \overset{b_{n-1}'}{\rightarrow} .... \overset{b_{0}'}{\rightarrow} A \rightarrow 0
	$$
	is a $\rh$-projective resolution of $A$ of length $n$.  Therefore $pd_{\rh}(A)\leq n$.  
	
	(3 $\Leftrightarrow$ 4) By proposition~\ref{propext11} there are isomorphism natural in $M$: 
	$$
	(\forall M \in _{A^e}Mod)\, HH^{1+n}(A,M) \cong Ext_{\rh}^{1+n}(A,M) \cong Ext_{\rh}^{1}(\nf,M).  
	$$
	Therefore for every $A^e$-module $M$: 
	$$
	Ext_{\rh}^{1}(\nf,M)\cong 0 \mbox{ if and only if } HH^{1+n}(A,M)\cong 0.  
	$$        By proposition~\ref{projcharext} $\nf$ is $\ea$-projective if and only if $
	Ext_{\rh}^{1}(\nf,M)\cong 0.$
	
	(2 $\Rightarrow$ 1)
	If $A$ admits an $\rh$-projective resolution $P_{\star}$ of length $n$ then theorem~\ref{thmcomparisonrelcohomolg} implies there are natural isomorphisms of $A^e$-modules: \begin{equation}
	(\forall M \in _{A^e}Mod) Ext_{\rh}^{\star}(A,M) \cong H^{\star}(Hom_{A^e}(P_{\star}, M)).  
	\label{ememfeoeomfoemf3}
	\end{equation}
	Since $P_{\star}$ is of length $n$ all the maps $p_j: P_{j+1}\rightarrow P_j$ are the zero maps therefore so are the maps $p_j^{\star}: Hom_{A^e}(P_{j})\rightarrow Hom_{A^e}(P_{j+1})$.  Whence~\eqref{ememfeoeomfoemf3} entails that for all $j > n+1$ $Ext_{\rh}^{\star}(A,M)$ vanishes.  By proposition~\ref{hochrelderived} this is equivalent to $HH^{j}(A,M)$ vanishing for all $j>n+1$ for all $M \in _{A^e}Mod$.  Hence $A$ is of Hochschild cohomological dimension at most $n$.  
\end{proof}
Next, the non-commutative geometric object focused on in this paper is reviewed.  
\subsection{Quasi-Free Algebras}
Many of the properties of an algebra are summarized by its Hochschild cohomological dimension, see \cite{HKR,beck1967triples,kratsios2015bounding} for example.  However, this article focuses on the following non-commutative analogue of smoothness of~\cite{grothendieck2002rev}, introduced by~\cite{cuntz1995algebra}\footnote{First introduced by Cuntz and Quillen in \cite{cuntz1995algebra}, due to their lifting property the quasi-free $k$-algebras are considered a non-commutative analogue to smooth $k$-algebras; that is $k$-algebras for which $\Omega_{A|k}$ is a projective $A$-module.  }.  This notion of smoothness has played a key role in a number of places in noncommutative algebraic geometry, especially in the cyclic (co)homology of \cite{connes1831noncommutative}.  
\begin{defn}[Quasi-free $k$-algebra]\label{defnquasifreecqW}
	A $k$-algebra for which all $k$-Hochschild extensions of $A$ by an $(A,A)$-bimodule lift is called a \textbf{quasi-free $k$-algebra}.  
\end{defn}
\begin{cor}\label{propQF}    
	For a $k$-algebra $A$, the following are equivalent:
	\begin{enumerate}
		\item $A$ is $HCdim(A|k)\leq 1$.  
		\item $\ncof$ is a $\rh$-projective $A^e$-module. 
		\item $A$ is quasi-free.  
	\end{enumerate}
\end{cor}
One typically construct quasi-free algebras using Morita equivalences.  However, the next proposition, which extends a result of \cite{cuntz1995algebra} to the case where $k$ need not be a field, may also be used without any such restrictions on $k$.  
\begin{prop}\label{lemqfromfqfandprojrel1}
	If $A$ is a quasi-free $k$-algebra and $P$ is an $\rh$-projective $(A,A)$-bimodule then $T_A(P)$ is a quasi-free $A$-algebra.  
\end{prop}
\begin{proof}
	Differed until the appendix.  
\end{proof}
\begin{ex}$\label{exqffreeeeeee1}$
	Let $n \in \mathbb{N}$.  The $\zz$-algebra $T_{\zz}\left(\underset{i=0}{\overset{n}{\bigoplus}} \zz\right)$ is quasi-free.  
\end{ex}
\begin{proof}
	Since all free $\zz$-modules are projective $\zz$-modules and all projective $\zz$-modules are $\ezz$-projective modules, the free $\zz$-module $\underset{i=0}{\overset{n}{\bigoplus}} \zz$ is $\ezz$-projective.  Whence proposition~\ref{lemqfromfqfandprojrel1} implies $T_{\zz}\left(\underset{i=0}{\overset{n}{\bigoplus}} \zz\right)$ is a quasi-free $\zz$-algebra.  
\end{proof}

\begin{ex}
	If $A$ is a quasi-free $k$-algebra then $T_A(\ncof)$ is a quasi-free $A$-algebra.  
\end{ex}
\begin{proof}
	By corollary~\ref{propQF} if $A$ is quasi-free $\ncof$ must be an $\rh$-projective $(A,A)$-bimodule; whence proposition~\ref{lemqfromfqfandprojrel1} applies.  
\end{proof}

Next, we overview some relevant dimension-theoretic notions and terminology.  
\subsection{Classical Cohomological Dimensions}
We remind the reader of a few important algebraic invariants which we will require.  The reader unfamiliar with certain of these notions from commutative algebra and algebraic geometry is referred to \cite{eisenbud1995commutative,hartshorne1977algebraic} or to \cite{stacks-project}.  

\begin{defn}[$A$-Flat Dimension]\label{defn:fd}
	If $A$ is a commutative ring then the $A$-flat dimension $fd_{A}(M)$ of an $A$-module $M$ is the extended natural number $n$, defined as the shortest \textit{length} of a resolution of $M$ by $A$-flat $A$-modules. If no such finite $n$ exists $n$ is taken to be $\infty$.   
\end{defn}
We will require the following result, whose proof can be found in \cite{weibel1995introduction,kratsios2015bounding}.  
\begin{prop}\label{propflatcalc}
	If $n$ is a positive integer and if there exists a  regular sequence $x_1,..,x_n$ in $A$ of length $n$ then: \begin{equation}
	n = fd_A(A/(x_1,..,x_n)).  
	\end{equation}
\end{prop}
One more ingredient related to the flat dimension will soon be needed.  
\begin{prop} $\label{fdloc}$
	If $A$ is a commutative ring and ${\mathfrak{m}}$ is a maximal ideal of $A$ then for any $A$-module $M$ $fd_{{{A_{\mathfrak{m}}}}}(M_{\mathfrak{m}})$ is a lower-bound for $fd_{A}(M)$.  
\end{prop}
%\begin{proof}
%    \begin{description}
%        \item[Case 1: $fd_{A}(M)$ is finite]   \hfill \\
%        \begin{enumerate}
%            \item Let $d$ be the ${A}$-flat dimension of $M$.  By definition, there is a deleted ${A}$-flat resolution $\textbf{F}_{\star}$ of $M$ of length $d$.  Since localization is exact \cite{dummit2004abstract}, 
%            $A_{\mathfrak{m}}\otimes_A F_{\star}$ is an exact sequence augmentable to $A_{\mathfrak{m}} \otimes_A M \cong M_{\mathfrak{m}}$.  
%            \item Again since localization is exact, $A_{\mathfrak{m}}$ is a flat $A$-module.  Since the tensor product of flat modules is again flat \cite{rotman1979introduction} each $A_{\mathfrak{m}} \otimes_A F_i$ in $A_{\mathfrak{m}} \otimes_A \textbf{F}_{\star}$ is flat as an $A_{\mathfrak{m}}$-module.  
%            \item Therefore, $A_{\mathfrak{m}} \otimes_A F_i$ is an $A_{\mathfrak{m}}$-flat resolution of $M_{\mathfrak{m}}$ of length $d$.  Whence, by definition the $A$-flat dimension of $M_{\mathfrak{m}}$ can therefore be at most $d$.  
%        \end{enumerate}
%        \item[Case 2: $fd_{A}(M)$ is infinite]  \hfill \\
%        By definition of $fd_{{A_{\mathfrak{m}}}}(M_{\mathfrak{m}})$: \begin{equation}
%        fd_{{A_{\mathfrak{m}}}}(M_{\mathfrak{m}}) \leq \infty = fd_{A}(M).
%        \end{equation}  
%    \end{description}
%\end{proof}
\begin{defn} $\label{defn:pd}$
	\textbf{${A}$-Projective Dimension}
	
	If $A$ is a commutative ring and $M$ is an $A$-module then the ${A}$-projective dimension $pd_{A}(M)$ of $M$ is the extended natural number $n$, defined as the shortest \textit{length} of a deleted ${A}$-projective resolution of $M$.  If no such finite $n$ exists $n$ is taken to be $\infty$.   
\end{defn}
\begin{lem} $\label{lem:fdpd}$
	
	If $A$ is a commutative ring and $M$ is an $A$-module then $fd_A(M)\leq pd_A(M)$.  
\end{lem}
\begin{proof}
	Since all $A$-projective $A$-modules are $A$-flat, then any $A$-projective resolution is a $A$-flat resolution.  
\end{proof}

\begin{lem} $\label{char:pd}$
	
	If $A$ is a commutative ring then for any $A$-module $M$ the following are equivalent: \begin{itemize}
		\item The ${A}$-projective dimension of $M$ is at most $n$.
		\item For every $A$-module $N$, the $A$-module $Ext_{n+1}^A(M,N)$ is trivial.   
		\item For every $A$-module $N$ and every integer $m\geq n+1$: $Ext_{m}^A(M,N)\cong 0$.     
	\end{itemize}
\end{lem} 
\begin{proof}
	Nearly identical to the proof of theorem~\ref{thrmcharnsmooth}, see page 456 of \cite{rotman1979introduction} for details. 
\end{proof}
\begin{defn}[Cohen-Macaulay at an Ideal]\label{def:CM}
	A commutative ring $A$ is said to be {Cohen-Macaulay at} a maximal ideal $\ma$ if and only if either: 
	\begin{itemize}
		\item $Krull(A_{\ma})$ is finite and there is an $\ama$-regular sequence $x_1,...,x_d$ in $A_{\ma}$ of maximal length $d=Krull(A_{\ma})$ such that $\{ x_1,..,x_d \} \subseteq \ma $.  
		\item $Krull(A_{\ma})$ is infinite and for every positive integer $d$ there is an $\ama$-regular sequence $x_1,..,x_d$ in $\ma$ on $A$ of length $d$.  
	\end{itemize}
\end{defn}
\begin{prop}[\cite{weibel1995introduction}]\label{prop:Krullpd}
	If $A$ is a commutative ring which is Cohen Macaulay at the maximal ideal $\mathfrak{m}$ and $Krull(A_{\ma})$ is finite then: \begin{equation}
	Krull(A_{\ma}) = fd_{A_{\ma}}(A_{\ma}/(x_1,..,x_n)) \leq pd_A(A_{\ma}/(x_1,..,x_n))
	\end{equation}
\end{prop}
%\begin{proof}
%    Since $A$ is Cohen-Macaulay at the maximal ideal $\ma$, there is a regular sequence $x_1,..,x_n$ in $\ma$ of length $n=Krull(A_{\ma})$.  
%    Denote $A_{\ma}/(x_1,..,x_n)$ by $\dm$.  
%    By corollary~\ref{lem:fd}: \begin{equation}
%    Krull(A_{\ma}) = fd_{A_{\ma}}(\dm).  
%    \label{mghfghmemefogdgeq11111}
%    \end{equation}
%    
%    Proposition~\ref{fdloc} applied to~\eqref{mghfghmemefogdgeq11111} entails:
%    \begin{equation}
%    Krull(A_{\mathfrak{m}}) 
%    = fd_{A_{\ma}}(\dm) \leq fd_{A}(\dm)
%    \label{fgfgdfgr45g54g54g}
%    \end{equation}
%    Lastly lemma~\ref{lem:fdpd} bounds~\eqref{fgfgdfgr45g54g54g} above as follows: \small
%    \begin{equation}
%    Krull(A_{\mathfrak{m}}) 
%    = fd_{A_{\mathfrak{m}}}(\dm) \leq fd_{A}(\dm) 
%    \leq pd_{A}(\dm) 
%    \label{g4g45g4g4g45g}. 
%    \end{equation}\normalsize
%\end{proof}
\begin{defn} $\label{defn:wDim}$
	\textbf{Global Dimension}
	
	The \textbf{global dimension} $D(A)$ of a ring $A$, is defined as the supremum of all the $A$-projective dimensions of its $A$-modules.  That is: \begin{equation}
	D(A):=\underset{M\in _AMod}{sup} pd_{A}(M).  
	\end{equation}  
\end{defn}
The following modification of the global dimension of a $k$-algebra, does not ignore the influence of $k$ on a $k$-algebra $A$, as will be observed in the next section of this paper.  
\begin{defn} $\label{defn:wDim}$
	\textbf{${\mathscr{E}^k}$-Global dimension}
	
	The ${\mathscr{E}^k}$-\textbf{global Dimension} $D_{{\mathscr{E}^k}}(A)$ of a $k$-algebra $A$ is defined as the supremum of all the ${\mathscr{E}_A^k}$-projective dimensions of its $A$-modules.  That is: \begin{equation}
	D_{{\mathscr{E}^k}}(A):=\underset{M\in _AMod}{sup} pd_{{\mathscr{E}_A^k}}(M).  
	\end{equation}  
\end{defn}
We may now formulate the central result of this paper and a few of its consequences.  
\section{A lower bound for the Hochschild cohomological dimension}
From hereon out, $A$ will always be a commutative $k$-algebra.  The remainder of this paper will focus on establishing the following result.  An analogous statement was made in \cite{cuntz1995algebra} that all affine algebraic varieties over $\cc$ of dimension at greater than $1$ fail to have a quasi-free $\cc$-algebra of functions.  Once, the assumption that $k=\cc$ is relaxed, we find an analogous claim is true; however, the analysis is more delicate.  Our finding is the following and central results of this paper.  
\begin{thrm}\label{theorem1A}
	Let $A$ be a commutative $k$-algebra and $\ma$ be a non-zero maximal ideal in $A$ such that $\ama$ is has finite $\kiki$-flat dimension and $D(\kiki )$ is finite.  
	\begin{enumerate}
		\item For every $A$-module $M$ there is a string of inequalities:  
		\small
		\begin{equation*}
		\begin{aligned}
		fd_{\ama}(M_{\ma}) -D(\kiki) 
		-fd_{k_{\ma}}(A_{\ma})
		\leq & pd_{\ama}(M_{\ma}) -D(\kiki) 
		-fd_{k_{\ma}}(A_{\ma})
		\\
		\leq& pd_{\eal}(M_{\ma}) \\
		\leq& pd_{\ea}(M)
		\leq D_{\ek}(A) \\
		\leq & HCdim(A|k)
		\end{aligned}
		\end{equation*}
		\normalsize
		\item If $A$ is Cohen-Macaulay at some maximal ideal $\ma$
		
		Then $Krull(A_{\ma}) - D(\kiki) -fd_{\ma}(A_{\ma}) \leq HCdim(A|k)$. 
		
		\textit{In this scenario: if $A_{\ma}$ is of Krull dimension at least $2+D(\kiki)-fd_{\ma}(A_{\ma})$ then $A$ is not Quasi-free}.  
	\end{enumerate}
\end{thrm}
Theorem~\ref{theorem1A} allows for an easily computable lower-bound on the Hochschild cohomological dimension of nearly any commutative $k$-algebra $A$, granted that it is smooth in the classical sense at-least at one point.  The next result, obtains an even simpler criterion under the additional assumption that $A$ is $k$-flat.  
\begin{thrm}\label{thrm_wkD1Ctada}
	Let $k$ be of finite global dimension, $A$ be a $k$-algebra which is flat as a $k$-module.  
	\begin{enumerate}
		\item For every $A$-module $M$ there is a string of inequalities:  
		\begin{equation}
		fd_{A}(M) -D(k) \leq pd_A(M) -D(k) \leq pd_{\ea}(M)\leq D_{\ek}(A) \leq HCdim(A|k)
		\label{eq1dcerion1}
		\end{equation}
		\item If $x_1,..,x_n$ is a  regular sequence in $A$ then:
		\begin{equation}
		n -D(k) \leq HCdim(A|k).  
		\label{eq2dcerion2}
		\end{equation}
		\textit{In this scenario: if $n$ is at least $2+D(k)$ then $A$ is not Quasi-free}.  
		\item If $\ncof$ is generated by a  regular sequence $x_1,...,x_n$ in $A\otimes_k A$ then:
		\begin{equation}
		n -D(k) \leq HCdim(A|k).  
		\end{equation}
		\textit{In this scenario: if $n$ is at least $2+D(k)$ then $A$ is not Quasi-free}.
		\item Furthermore if $A$ is commutative and Cohen-Macaulay at a maximal ideal $\mathfrak{m}$ then:
		\begin{equation} Krull(A_{\mathfrak{m}}) - D(k) \leq HCdim(A|k). 
		\end{equation}
		\textit{In this scenario: if $A_{\mathfrak{m}}$'s Krull dimension is at least $2+D(k)$ then $A$ is not Quasi-free}.  
	\end{enumerate}
\end{thrm}
%The proof of Theorem~\ref{thrm_wkD1Ctada} will follow from a sequence of Lemmas; the first of which generalizes a result of~\cite{hochschild1958note}.  
%
%
%We may now provide a lower-bound for the Hochschild Cohomological dimension of certain commutative $k$-algebras.  The argument revolves around Hochschild dimension of a regular commutative $k$-algebra (to be defined below in definition~\ref{defnreg}) below via a series of intermediary numerical invariants associated to the algebra $A$.  
%\begin{thrm}\label{thrmHochschild1958old}
%    \textbf{Hochschild} ($\sim$1958)
%    
%    If $k$ is of finite global dimension, $A$ is a $k$-algebra which is flat as a $k$-module and $M$ is an $A$-module then: \begin{equation}
%    pd_A(M) - D(k) \leq pd_{\ea}(M)
%    \label{frb54yn56yne6u5e8umdu6mhhhhhh}
%    \end{equation}
%\end{thrm}
%\begin{proof}
%    See \cite{hochschild1958note}.  
%\end{proof}
Before moving to the proof of these results, let us consider a counter-intuitive consequence.  Namely, that most examples of smooth commutative algebras fail to be quasi-free, even when $k\neq \cc$.  This makes smoothness, in the sense of \cite{cuntz1995algebra}, very rare in the non-commutative category.  The following example from arithmetic geometry is of interest.  
\begin{cor}[Arithmetic Polynomial-Algebras]
	The $\mathbb{Z}$-algebra $\mathbb{Z}[x_1,..,x_n]$ fails to be quasi-free for values of $n>1$.  
\end{cor}
\begin{proof}
	Since $\mathbb{Z}[x_1,...x_n]$ is Cohen-Macaulay at the maximal ideal $(x_1,...x_n,p)$ and is of Krull dimension $n+1=Krull(\mathbb{Z}[x_1,...x_n])$.  Moreover, one computes that $D(\mathbb{Z})=1$.  Whence by point $2$ of theorem~\ref{thrm_wkD1Ctada}: $\mathbb{Z}[x_1,..,x_n]$ fails to be Quasi-free if $2\leq Krull(\mathbb{Z}[x_1,...x_n]) - D(\mathbb{Z}) = (n+1)-1 =n$.  
\end{proof}
\subsection{{Proof of Theorem~\ref{theorem1A}}}
Our first lemma is a generalization of the central theorem of \cite{hochschild1958note}; which does not rely on the assumption that $A$ is $k$-flat.  
\begin{lem}\label{thrmHochschild1958}
	If $k$ is of finite global dimension and $A$ is a $k$-algebra which is of finite flat dimension as a $k$-module, then for every $A$-module $M$: \begin{equation}
	pd_A(M) - D(k) -fd_k(A) \leq pd_{\ea}(M)
	\label{frb54yn56yne6u5e8umdu6mhhhhhh}
	\end{equation}
\end{lem}
%-----------000000000000000---------------
%  Here
%-----------000000000000000---------------
The proof of lemma~\ref{thrmHochschild1958} relies on the following lemma.
\begin{lem}$\label{hlem1}$
	If $A$ is a $k$-algebra such that $fd_k(A)<\infty$ then: \begin{equation}
	( \forall M \in _kMod ) \mbox{ } pd_A(A\otimes_k M) -fd_k(A) \leq pd_{k}(M)
	\end{equation}
	%Generalize the second part of the theorem assuming that pd_k(A) is finite :)
\end{lem}
\begin{proof}
	%Maybe make a note on A\otimes_k P being projective if P was k-projective to start off with.  
	For every $k$-module $M$ and every $A$-module $N$ there is a convergent third quadrant spectral sequence:
	\begin{equation}
	Ext^p_A(Tor_q^k(A,M),N)\underset{p}{\Rightarrow} Ext_k^{p+q}(M,Hom_A(A,N))
	\label{specsec1}
	\mbox{\cite{rotman1979introduction}}
	%[pg. 667]
	.  
	\end{equation}
	Moreover the adjunction $-\otimes_k A \adj Hom_A(A,-)$ extends to a natural isomorphism:
	\begin{equation}
	(\forall p,q \in \mathbb{N}) Ext_k^{p+q}(M,Hom_A(A,N))\cong Ext_A^{p+q}(M\otimes_k A, N) \mbox{\cite{weibel1995introduction}}.  
	\end{equation}
	Therefore there is a convergent third-quadrant spectral sequence:
	\begin{equation}
	Ext^p_A(Tor_q^k(A,M),N)\underset{p}{\Rightarrow} Ext_A^{p+q}(M\otimes_k A, N)
	\label{specsec2}.  
	\end{equation}
	If $pd_A(N)<\infty$, then the result is immediate.  Therefore assume that:
	$pd_A(N)<\infty$.  If $p+q>fd_k(A)+pd_A(N)$ then either $p>pd_A(N)$ or $q>fd_k(A)$.  In the case of th\begin{equation*}
	0\cong E_2^{p,q}\cong E_{\infty}^{p,q}\cong Ext_A^{p+q}(M\otimes_k A, N)
	\end{equation*} and in the latter case \begin{equation*}
	0 \cong E_2^{p,q}\cong E_{\infty}^{p,q}\cong Ext_A^{p+q}(M\otimes_k A, N)
	\end{equation*} also.  Therefore \begin{equation*}
	(\forall N\in A_Mod) \mbox{ } 0\cong Ext_A^{n}(M\otimes_k A, N) \mbox{if $n>fd_k(A)+pd_A(N)$; }
	\end{equation*}hence:
	$pd_A(M\otimes_k A)\leq fd_k(A)+pd_A(M)$.  
	
	Finally, the result follows since $fd_k(A)$ is finite and, therefore, can be subtracted unambiguously.  
\end{proof}

%In order to make use of lemma $\autoref{hlem1}$ from \textit{hereon end} the following assumption will be made of all $k$-algebras:
%\begin{ass}
%Every $k$-algebra $A$ is flat as a $k$-module.  
%\end{ass}

\begin{lem}$\label{hlem2}$
	If $A$ is a $k$-algebra then for any $k$-module $M$ there is an $\ea$-exact sequence: \begin{equation}
	\begin{tikzpicture}[>=angle 90]
	\matrix(a)[matrix of math nodes,
	row sep=3em, column sep=3em,
	text height=1.5ex, text depth=0.25ex]
	{0 & Ker(a) & A \otimes_k M & M & 0\\};
	\path[->,font=\scriptsize]
	(a-1-1) edge node[above]{$  $} (a-1-2);
	\path[->,font=\scriptsize]
	(a-1-2) edge node[above]{$  $} (a-1-3);
	\path[->,font=\scriptsize]
	(a-1-3) edge node[above]{$ \alpha $} (a-1-4);
	\path[->,font=\scriptsize]
	(a-1-4) edge node[above]{$  $} (a-1-5);
	\end{tikzpicture}
	\end{equation}
	Where $\alpha$ be the map defined on elementary tensors $(a\otimes_k m)$ in $A\otimes_k M$ as $a\otimes_k m \mapsto a\cdot m$.  
\end{lem}
\begin{proof}
	$\alpha$ is $k$-split by the map $\beta : M \rightarrow A \otimes_k M$ defined on elements $m\in M $ as $m \mapsto 1 \otimes_k m$.  Indeed if $m \in M$ then: \begin{equation}
	\alpha \circ \beta (m) = \alpha (1 \otimes_k m) = 1 \cdot m = m.  
	\end{equation}
\end{proof}
\begin{lem}$\label{lemdirsum}$
	If $M$ and $N$ are $A$-modules then: \begin{equation}
	pd_A(M)\leq pd_A(M\oplus N).  
	\end{equation}
\end{lem}
\begin{proof}
	\begin{equation}
	(\forall n \in \mathbb{N})(\forall X \in _AMod) \mbox{ } Ext^n_A(M,X) \oplus Ext^n_A(N,X) \cong Ext^n_A(M\oplus N,X). 
	\end{equation}  
	Therefore $Ext^n_A(M\oplus N,X)$ vanishes only if both $Ext^n_A(M,X)$ and $Ext^n_A(N,X)$ vanish.  Lemma~\ref{char:pd} then implies: $pd_A(M)\leq pd(M \oplus N)$.  
\end{proof}

\paragraph{{Proof of lemma~\ref{thrmHochschild1958}}}
\begin{proof} \hfill
	\begin{description}
		\item[Case $1$: $pd_{\ea}(M) =\infty$] \hfill \\
		
		By definition $pd_A(M) \leq \infty $ therefore trivially if $ pd_{\ea}(M) = \infty$ then: \begin{equation}
		pd_A(M) \leq  pd_{\ea}(M) + D(k).  
		\label{4vb54tb54tb555555}
		\end{equation}
		Since $k$'s global dimension is finite hence~\eqref{4vb54tb54tb555555}
		implies: 
		\begin{equation}
		pd_A(M) - D(k) \leq \infty = pd_{\ea}(M).  
		\end{equation}
		
		\item[Case $2$: $pd_{\ea}(M) <\infty$] \hfill 
		
		Let \textbf{$d:=pd_{\ea}(M)+D(k)+fd_k(A)$}.  
		The proof will proceed by induction on \textbf{$d$}.  
		\begin{description}
			\item[Base: $d=0$] \hfill \\
			Suppose $pd_{\ea}(M)=0$.  
			
			By theorem~\ref{thrmcharnsmooth} $M$ is $\ea$-projective. 
			Lemma~\ref{hlem2} implies there is an $\ea$-exact sequence: \begin{equation}
			\begin{tikzpicture}[>=angle 90]
			\matrix(a)[matrix of math nodes,
			row sep=3em, column sep=3em,
			text height=1.5ex, text depth=0.25ex]
			{0 & Ker(\alpha ) & A \otimes_k M & M & 0\\};
			\path[->,font=\scriptsize]
			(a-1-1) edge node[above]{$  $} (a-1-2);
			\path[->,font=\scriptsize]
			(a-1-2) edge node[above]{$  $} (a-1-3);
			\path[->,font=\scriptsize]
			(a-1-3) edge node[above]{$ \alpha $} (a-1-4);
			\path[->,font=\scriptsize]
			(a-1-4) edge node[above]{$  $} (a-1-5);
			\end{tikzpicture}.  
			\label{firstsequnncetosplitand12}
			\end{equation}
			Proposition~\ref{eaprojchar} implies that~\eqref{firstsequnncetosplitand12} is $A$-split therefore $M$ is a direct summand of the $A$-module $A\otimes_k M$.   
			Hence lemma~\ref{lemdirsum} implies: \begin{equation}
			pd_A(M)\leq pd_A(M\otimes_k A) 
			% \mbox{ \cite{stacks-project}}
			.  
			\label{moretinhgs1111eq1}
			\end{equation}
			Lemma~\ref{hlem1} together with~\eqref{moretinhgs1111eq1} imply: \begin{equation}
			pd_A(M)\leq pd_A(M\otimes_k A) \leq pd_k(M).  
			\label{moretinhgs1111eq2}
			\end{equation}
			Definition~\ref{defn:wDim} and~\eqref{moretinhgs1111eq2} together with the assumption that $pd_{\ea}(M)=0$ imply: \begin{equation}
			pd_A(M) \leq pd_k(M) \leq D(k) = D(k) +0 +0 = D(k) + pd_{\ea}(M)+fd_k(A). 
			\label{efemeomoemeo}
			\end{equation}
			Since $k$'s global dimension and $fd_k(A)$ are finite then~\eqref{efemeomoemeo} implies: \begin{equation}
			pd_A(M) - D(k) -fd_k(A)\leq pd_{\ea}(M).  
			\end{equation}
			\item[Inductive Step: $d>0$] \hfill \\
			Suppose the result holds for all $A$-modules $K$ such that $pd_{\ea}(K)+D(k)+fd_k(A)=d$ for some integer $d>0$.  
			Again appealing to lemma~\ref{hlem2}, there is an $\ea$-exact sequence: \begin{equation}
			\begin{tikzpicture}[>=angle 90]
			\matrix(a)[matrix of math nodes,
			row sep=3em, column sep=3em,
			text height=1.5ex, text depth=0.25ex]
			{0 & Ker(\alpha ) & A \otimes_k M & M & 0\\};
			\path[->,font=\scriptsize]
			(a-1-1) edge node[above]{$  $} (a-1-2);
			\path[->,font=\scriptsize]
			(a-1-2) edge node[above]{$  $} (a-1-3);
			\path[->,font=\scriptsize]
			(a-1-3) edge node[above]{$ \alpha $} (a-1-4);
			\path[->,font=\scriptsize]
			(a-1-4) edge node[above]{$  $} (a-1-5);
			\end{tikzpicture}.  
			\label{firstsequnncetosplitand1}
			\end{equation}
			Proposition~\ref{eaprojchar} implies $A\otimes_kM$ is $\ea$-projective; whence~\eqref{firstsequnncetosplitand1} implies: \begin{equation}
			pd_{\ea}(Ker(\alpha )) +1  = pd_{\ea}(M).  
			\end{equation}
			Since $Ker(\alpha )$ is an $A$-module of strictly smaller $\ea$-projective dimension than $M$ the induction hypothesis applies to $Ker(\alpha )$ whence: \begin{equation}
			pd_{A}(Ker(\alpha )) +1 \leq pd_{\ea}(Ker(\alpha )) +1 + D(k) +fd_k(A)\leq pd_{\ea}(M) +D(k)+fd_k(A).  
			\label{mememememeeeeeqq1finaICXPNIKA}
			\end{equation}
			
			The proof will be completed by demonstrating that: $pd_A(M) \leq pd_A(Ker(\alpha)) +1$.

			For any $N\in _AMod$ $Ext_A^{\star}(-,N)$ applied to~\eqref{firstsequnncetosplitand1} gives way to the long exact sequence in homology, particularly the following of its segments are exact: 
			\small
			\begin{equation}
			\begin{tikzpicture}[>=angle 90]
			\matrix(a)[matrix of math nodes,
			row sep=1em, column sep=0.5em,
			text height=1.5ex, text depth=0.25ex]
			{Ext_A^{n-1}(A \otimes_k M,N) & Ext_A^{n-1}(Ker(a),N) &             Ext_A^{n}(M,N) & Ext_A^{n}(A \otimes_k M,N) \\};
			\path[->,font=\scriptsize]
			(a-1-1) edge node[above]{$  $} (a-1-2);
			\path[->,font=\scriptsize]
			(a-1-2) edge node[above]{$ \partial^n $} (a-1-3);
			\path[->,font=\scriptsize]
			(a-1-3) edge node[above]{$  $} (a-1-4);
			\end{tikzpicture}
			\label{FirstrefHdimproof1}
			\end{equation}
			\normalsize
			Since $A\otimes_k M$ is $\ea$-projective $pd_{\ea}(A \otimes_k M)=0$, therefore by the base case of the induction hypothesis $pd_A(A\otimes_k M) \leq pd_{\ea}+ D(k) +fd_k(A)= D(k)+fd_k(A)$; thus for every positive integer $n\geq D(k)$ (in particular $d$ is at least $n$): \begin{equation}
			(\forall N \in _AMod)\mbox{ } Ext_A^{n-1}(A\otimes_k M,N) \cong 0 \cong Ext_A^{n}(A\otimes_k M,N);  
			\label{FirstrefHdimproof2}
			\end{equation}
			whence $\partial^n$ must be an isomorphism.  
			Therefore lemma~\ref{char:pd} implies $pd_A(M)$ is at most equal to $pd_A(Ker( \alpha )) +1$.  
			
			Therefore: 
			\begin{align}
			&pd_A(M) \leq pd_{A}(Ker(\alpha )) +1 \\
			&\leq pd_{\ea}(Ker(\alpha )) +1 + D(k)+fd_k(A) \\
			&\leq pd_{\ea}(M) +D(k)+fd_k(A).  
			\label{mememememeeeeeqqfinaICXPNIKA}
			\end{align}
			Finally since $k$ is of finite global dimension and $A$ is of finite $k$-flat dimension then~\eqref{mememememeeeeeqqfinaICXPNIKA} implies: \begin{equation}
			pd_A(M) - D(k)-fd_k(A) \leq  pd_{\ea}(M); 
			\end{equation}
			thus concluding the induction.
		\end{description}
	\end{description}
\end{proof}
We will also require the following result.  
\begin{rremark}
	Let $A$ be a $k$-algebra, $i: k \rightarrow A$ the morphism defining the $k$-algebra $A$ and $\ma$ a maximal ideal in $A$.  For legibility the $\mathscr{E}_{A_{\ma}}^{k_{i^{-1}[\ma]}}$-projective dimension of an $A_{\ma}$-module $N$ will be abbreviated by $pd_{\eal}(N)$ (instead of writing $pd_{\eall}(N)$).  
\end{rremark}
\begin{lem}$\label{localbound}$
	If $A$ is a commutative $k$-algebra and $\ma$ is a non-zero maximal ideal in $A$ then for every $A$-module $M$: \begin{equation}
	pd_{\eal}(M_{\ma}) \leq pd_{\ea}(M),
	\end{equation}
	where $i: k \rightarrow A$ is the inclusion of $k$ into $A$.  
\end{lem}
\begin{proof}
	Since $\ma$ is a prime ideal in $A$, $i^{-1}[\ma]$ is a maximal ideal in $\kiki$, whence the localized ring $\kiki$ is a well-defined sub-ring of $A_{\ma}$.  
	Let
	\begin{equation}
	\begin{tikzpicture}[>=angle 90]
	\matrix(a)[matrix of math nodes,
	row sep=3em, column sep=2.5em,
	text height=1.5ex, text depth=0.25ex]
	{...\overset{d_{n+1}}{\rightarrow} P_n & ... & P_1 & P_0 & M & 0 \\};
	\path[->,font=\scriptsize]
	(a-1-1) edge node[above]{$ d_{n} $} (a-1-2);
	\path[->,font=\scriptsize]
	(a-1-2) edge node[above]{$ d_2 $} (a-1-3);
	\path[->,font=\scriptsize]
	(a-1-3) edge node[above]{$ d_1 $} (a-1-4);
	\path[->,font=\scriptsize]
	(a-1-4) edge node[above]{$ d_0 $} (a-1-5);
	\path[->,font=\scriptsize]
	(a-1-5) edge node[above]{$  $} (a-1-6);
	\end{tikzpicture} 
	\label{deltoexseq11}
	\end{equation} be an $\ea$-projective resolution of an $A$-module $M$. 
	The exactness of localization \cite{eisenbud1995commutative} implies: 
	\begin{equation}
	\begin{tikzpicture}[>=angle 90]
	\matrix(a)[matrix of math nodes,
	row sep=3em, column sep=3em,
	text height=1.5ex, text depth=0.25ex]
	{...\overset{d_{n+1}}{\rightarrow} P_n \otimes_A A_{\ma} & ... & P_1 \otimes_A A_{\ma} & P_0 \otimes_A A_{\ma} & M \otimes_A A_{\ma} \rightarrow 0 \\};
	\path[->,font=\scriptsize]
	(a-1-1) edge node[above]{$ d_{n} \otimes_A A_{\ma} $} (a-1-2);
	\path[->,font=\scriptsize]
	(a-1-2) edge node[above]{$ d_2 \otimes_A A_{\ma} $} (a-1-3);
	\path[->,font=\scriptsize]
	(a-1-3) edge node[above]{$ d_1 \otimes_A A_{\ma} $} (a-1-4);
	\path[->,font=\scriptsize]
	(a-1-4) edge node[above]{$ d_0 \otimes_A A_{\ma} $} (a-1-5);
	\end{tikzpicture} 
	\label{deltoexseq12}
	\end{equation} is exact.  
	It will now be verified that~\eqref{deltoexseq12} is a $\eal$-projective resolution of the $A_{\ma}$-module $M_{\ma}$.  
	\begin{description}
		\item[The $d_{n} \otimes_A A_{\ma}$ are $\kiki$-split] \hfill\\
		Since~\eqref{deltoexseq11} was $k$-split then for every $i \in \mathbb{N}$ there existed a $k$-module homomorphism $s_i: P_{n-1} \rightarrow P_n$ (where for convenience write $P_{-1}:=M$) satisfying $d_i = d_i \circ s_i \circ d_i$.  
		Since $A_{\ma}$ is a $\kiki$-algebra $A_{\ma}$ may be viewed as a $\kiki$-module therefore the maps: $s_i \otimes_A 1_{A_{\ma}}$ are $\kiki$-module homomorphisms; moreover they must satisfy: \begin{equation}
		d_i \otimes_A 1_{A_{\ma}}= d_i \otimes_A 1_{A_{\ma}}\circ s_i \otimes_A 1_{A_{\ma}}\circ d_i \otimes_A 1_{A_{\ma}}.  
		\end{equation}
		Therefore~\eqref{deltoexseq12} is $\kiki$-split-exact.  
		\item[The $P_i \otimes_A A_{\ma}$ are $\eal$-projective] \hfill\\
		For each $i \in \mathbb{N}$ if $P_i$ is $\ea$-projective therefore proposition~\ref{eaprojchar} implies there exists some $A$-module $Q$ and some $k$-module $X$ satisfying: \begin{equation}
		P_i \oplus Q \cong A \otimes_k X.
		\label{IXCPNIKA}
		\end{equation}
		Therefore: 
		\small
		\begin{equation*}
		(P_i \otimes_A A_{\ma}) \oplus (Q \otimes_A A_{\ma}) 
		\cong 
		(P_i \otimes_A Q) \otimes_A A_{\ma}
		\cong
		(A \otimes_k X) \otimes_A A_{\ma}
		\end{equation*}
		\begin{equation}
		\cong
		(A \otimes_k X) \otimes_A (A_{\ma} \otimes_{\kiki} \kiki)
		\label{eqfinastiff1}
		\end{equation}
		\normalsize
		Since $A,k$ and $\kiki$ are commutative rings the tensor products $-\otimes_A -$, $-\otimes_k -$ and $-\otimes_{\kiki} -$ are symmetric \cite{rotman1979introduction}, hence~\eqref{eqfinastiff1} implies: 
		\small
		\begin{equation*}
		(P_i \otimes_A A_{\ma}) \oplus (Q \otimes_A A_{\ma}) 
		\cong
		(A \otimes_k X) \otimes_A (A_{\ma} \otimes_{\kiki} \kiki)
		\end{equation*}
		\begin{equation}
		\cong (A_{\ma} \otimes_A A) \otimes_{\kiki} (\kiki \otimes_k X)
		\label{feoforir98g5j8g54g}
		\end{equation}
		Since $A$ is a subring of $A_{\ma}$ then~\eqref{feoforir98g5j8g54g} implies:
		\begin{equation}
		(P_i \otimes_A A_{\ma}) \oplus (Q \otimes_A A_{\ma}) 
		\cong
		A_{\ma} \otimes_{\kiki} (\kiki \otimes_k X).  
		\label{eqfinastiff2}
		\end{equation}
		\normalsize
		$(\kiki \otimes_k X)$ may be viewed as a $\kiki$-module with action $\hat{\cdot}$ defined as: \begin{equation}
		(\forall c \in k)(\forall (c'\otimes_k x) \in \kiki \otimes_k X) \mbox{ } c\hat{\cdot} (c'\otimes_k x) := c\cdot c' \otimes x.  
		\end{equation}
		Since $(\kiki \otimes_k X)$ is a $\kiki$-module then for each $i \in \mathbb{N}$ $(P_i \otimes_A A_{\ma})$ is a direct summand of an $A_{\ma}$-module of the form $A_{\ma} \otimes_{\kiki} X'$ where $X'$ is a $\kiki$-module, 
		thus proposition~\ref{eaprojchar} implies that $P_i \otimes_A A_{\ma}$ is $A_{\ma}$-projective.  
	\end{description}
	Hence~\eqref{deltoexseq12} is an $\eal$-projective resolution of $M \otimes_A A_{\ma} \cong M_{\ma}$; whence: \begin{equation}
	pd_{\eal}(M_{\ma}) \leq pd_{\ea}(M).  
	\end{equation}
\end{proof}
All the homological dimensions discussed to date are related as follows:
\begin{prop}\label{propprojtoGlb}
	If $A$ is a commutative $k$-algebra and $\ma$ be a non-zero maximal ideal in $A$ such that $\ama$ has finite $\kiki$-flat dimension and $D(\kiki )$ is finite then there is a string of inequalities:  
	\small
	\begin{equation*}
	fd_{A_{\ma}}(M_{\ma}) -D(\kiki) 
	-fd_k(A) 
	\leq pd_{A_{\ma}}(M_{\ma}) -D(\kiki)
	-fd_k(A)  \leq pd_{\eal}(M_{\ma}) \leq pd_{\ea}(M) \leq D_{\ek}(A)
	\end{equation*}
	\normalsize
\end{prop}
\begin{proof} \hfill 
	\begin{enumerate}
		\item By definition: $pd_{{\mathscr{E}_A^k}}(M) \leq D_{\ek}(A)$.
		\item By lemma~\ref{localbound}: $pd_{\eal}(M_{\ma}) \leq pd_{\ea}(M)$
		\item Since $A_{\ma}$ is flat as a $\kiki$-module and $D(\kiki )$ is finite lemma~\ref{thrmHochschild1958} entails: \\
		$pd_{A_{\ma}}(M_{\ma})-D(\kiki) -fd_k(A) \leq pd_{\eal}(M_{\ma})$
		\item Lemma~\ref{lem:fdpd} implies: 
		\begin{equation}
		fd_{A_{\ma}}(M_{\ma}) \leq pd_{A_{\ma}}(M_{\ma}).  
		\label{fer4f444}
		\end{equation}
		Since the global dimension of $\kiki$ was assumed to be finite~\eqref{fer4f444} implies: 
		\begin{equation}
		fd_{A_{\ma}}(M_{\ma}) -D(\kiki) \leq pd_{A_{\ma}}(M_{\ma}) - D(\kiki).  
		\end{equation}
	\end{enumerate}
\end{proof}

\begin{lem} $\label{lem:pre-weakWeibeliso}$
	If $A$ is a commutative $k$-algebra and $M$ and $N$ be $A$-modules, then there are natural isomorphisms: \begin{equation}
	Ext_{\ea}^n(M,N) \cong HH^n(A,Hom_k(M,N)) \cong Ext_{\rh}^n(A,Hom_k(M,N)).  
	\end{equation}
\end{lem}
\begin{proof} \hfill \\
	\begin{itemize}
		\item
		For any $(A,A)$-bimodule $X$, $X\otimes_A M$ is an $(A,A)$-bimodule \cite{rotman1979introduction}[Cor. 2.53].  
		
		\item Moreover there are natural isomorphisms: \begin{equation}
		Hom_{_AMod}(X \otimes_A M, N) \overset{\cong}{\rightarrow} Hom_{_AMod_A}(X, Hom_{_kMod}(M, N)) \mbox{ \cite{rotman1979introduction}[Thrm. 2.75]}.  
		\label{remgfierjg9843980tj4444}
		\end{equation}
		In particular~\eqref{remgfierjg9843980tj4444} implies that for every $n$ in $\mathbb{N}$ there is an isomorphism which is natural in the first input: \begin{equation}
		Hom_{_AMod}(A^{\otimes n} \otimes_A M, N) \overset{\psi_n}{\rightarrow} Hom_{_AMod_A}(A^{\otimes n}, Hom_{_kMod}(M, N)).  
		\end{equation}
		Whence if $b'_{n+1}:A^{\otimes n+3} \rightarrow A^{\otimes n+2}$ is the $n^{th}$ map in the Bar complex (recall example ~\ref{leembares1}) and for legibility denote $Hom_{_AMod_A}(b'_n,Hom_k(M,N))$ by $\beta_n$.  The naturality of the maps $\psi_n$ imply the following diagram of $k$-modules commutes: \begin{equation}
		\begin{tikzpicture}[>=angle 90]
		\matrix(a)[matrix of math nodes,
		row sep=2.5em, column sep=4em,
		text height=1.5ex, text depth=0.25ex]
		{ Hom_{_AMod}(A^{\otimes n+2} \otimes_A M, N) & Hom_{_AMod_A}(A^{\otimes n+2}, Hom_{_kMod}(M, N)) \\
			Hom_{_AMod}(A^{\otimes n+3} \otimes_A M, N) & Hom_{_AMod_A}(A^{\otimes n+3}, Hom_{_kMod}(M, N)) \\};
		\path[=,font=\scriptsize]
		(a-1-1) edge [double] node[above]{$ \psi_n $} (a-1-2);
		\path[=,font=\scriptsize]
		(a-2-1) edge [double] node[above]{$ \psi_{n+1} $} (a-2-2);
		\path[->,font=\scriptsize]
		(a-1-1) edge node[left]{$ \psi_{n+1}^{-1}\circ \beta_n \circ \psi_n $}(a-2-1);
		\path[->,font=\scriptsize]
		(a-1-2) edge node[right]{$ \beta_n $}(a-2-2);
		\end{tikzpicture}
		\label{comdiagmnatuICXPNIKA}.  
		\end{equation}
		\item Therefore for every $n$ in $\mathbb{N}$: \begin{equation*}
		(\psi_{n+2}^{-1}\circ \beta_{n+1} \circ \psi_{n+1}) \circ (\psi_{n+1}^{-1}\circ \beta_n \circ \psi_n ) 
		\end{equation*}
		\begin{equation}
		= \beta_{n+1} \circ \beta_n =0.  
		\end{equation}
		Whence $<Hom_{_AMod}(A^{\otimes \star+2} \otimes_A M, N), (\psi_{\star+1}^{-1}\circ \beta_{\star} \circ \psi_{\star} )>$ is a chain complex.  Moreover the commutativity of~\eqref{comdiagmnatuICXPNIKA} implies: \begin{equation*}
		(\forall n\in \mathbb{N}) \mbox{ } H^{n}(Hom_{_AMod}(A^{\otimes \star+2} \otimes_A M, N)) = Ker(\psi_{\star+1}^{-1}\circ \beta_{\star} \circ \psi_{\star} )/Im(\psi_{n+2}^{-1}\circ \beta_{n+1} \circ \psi_{n+1})
		\end{equation*}
		\begin{equation*}
		\cong Ker(\beta_n)/Im(\beta_{n+1}) = H^{n}(Hom_{_AMod_A}(A^{\otimes \star+2}, Hom_{_kMod}(M, N))).  
		\label{ef43f43cf4vc44vc44422221ICXPNIKA}
		\end{equation*}
		\begin{equation}
		= HH^{n}(A,Hom_k(M,N))
		\end{equation}
		Furthermore proposition~\ref{hochrelderived} implies there are natural isomorphisms: \begin{equation}
		HH^{n}(A,Hom_k(M,N)) \cong Ext_{\rh}^{n}(A,Hom_k(M,N));
		\end{equation}
		Whence for all $n$ in $\mathbb{N}$ there are natural isomorphisms: 
		\footnotesize
		\begin{equation}
		H^{n}(Hom_{_AMod}(A^{\otimes \star+2} \otimes_A M, N)) \cong HH^{n}(A,Hom_k(M,N)) \cong Ext_{\rh}^{n}(A,Hom_k(M,N)).  
		\end{equation}
		\normalsize
		\item Finally if $M$ is an $A$-module then $<Hom_{_AMod}(A^{\otimes \star+2} \otimes_A M, N), (\psi_{\star+1}^{-1}\circ \beta_{\star} \circ \psi_{\star} )>$ calculates the $\ea$-relative Ext groups of $M$ with coefficients in $N$; therefore there are natural isomorphisms: \begin{equation}
		H^{n}(Hom_{_AMod}(A^{\otimes \star+2} \otimes_A M, N)) \cong Ext_{\ea}^n(M,N) \mbox{ \cite{weibel1995introduction}[pg. 289]}.  
		\end{equation}
		\item Putting it all together, for every $n$ in $\mathbb{N}$ there are natural isomorphisms: \begin{equation}
		Ext_{\rh}^{n}(A,Hom_k(M,N)) \cong HH^{n}(A,Hom_k(M,N)) \cong Ext_{\rh}^{n}(A,Hom_k(M,N)).  
		\end{equation}
	\end{itemize}
\end{proof}
We may now prove theorem~\ref{theorem1A}.  
\begin{proof}[{Proof of Theorem~\ref{theorem1A}}]
	\begin{enumerate}
		\item For any $A$-modules $M$ and $N$ lemma~\ref{lem:pre-weakWeibeliso} implied:
		\begin{equation}
		Ext_{\ea}^{\star}(N,M) \cong HH^{\star}(A,Hom_k(N,M))
		\label{gfghr4eu9gh948h48gjg}.
		\end{equation}  Therefore taking supremums over all the $A$-modules $M,N$, of the integers $n$ for which~\eqref{gfghr4eu9gh948h48gjg} is non-trivial implies:
		\small\begin{equation}
		D_{{\mathscr{E}^k}}(A)= \underset{M,N\in _AMod}{sup} (sup(\{n \in \mathbb{N}^{\#} | Ext^n(M,N)\neq 0 \})) \end{equation}
		\begin{equation}
		= \underset{M,N\in _AMod}{sup} (sup(\{n \in \mathbb{N}^{\#} | HH^{n}(A,Hom_k(N,M))\neq 0 \}))
		\label{dgfohrg9hg943hg0934jg}.  
		\end{equation}\normalsize  $Hom_k(N,M)$ is only a particular case of an $A^e$-module; therefore taking supremums over \textit{all} $A$-modules bounds~\eqref{dgfohrg9hg943hg0934jg} above as follows:
		\small
		\begin{equation}
		D_{{\mathscr{E}^k}}(A)
		= \underset{M,N\in _AMod}{sup} (sup(\{n \in \mathbb{N}^{\#}| HH^{\star}(A,Hom_k(N,M))\neq 0 \})) \end{equation}
		\begin{equation}
		\leq 
		\underset{\tilde{M}\in _{A^e}Mod}{sup} (sup(\{n \in \mathbb{N}^{\#} | HH^{n}(A,\tilde{M})\neq 0 \})). \label{gfdgg54by65yb6e5ybe65}
		\end{equation}\normalsize  
		The right hand side of~\eqref{gfdgg54by65yb6e5ybe65} is precisely the definition of the Hochschild cohomological dimension.  Therefore \begin{equation}
		D_{\ek}(A)\leq HCdim(A|k)
		\label{eqeee1i1ii1i1i1i1ij2iu32u9h439iuhvu93i344445}
		\end{equation}
		
		Proposition~\ref{propprojtoGlb} applied to~\eqref{eqeee1i1ii1i1i1i1ij2iu32u9h439iuhvu93i344445}, which draws out the conclusion.  
		\item 
		\begin{description}
			\item[Case 1: $Krull(A_{\ma})$ is finite] \hfill \\
			Since $A$ is Cohen-Macaulay at $\ma$ there is an $\ama$-regular sequence $x_1,..,x_d$ in $\ma$ of length $d:=Krull(A_{\ma})$ in $A_{\ma}$.  Therefore proposition~\ref{propflatcalc} implies:
			\begin{equation}
			Krull(A_{\ma}) = fd_{A_{\ma}}(A_{\ma}/(x_1,..,x_n)).   
			\label{ewrkwwkkrewgm}
			\end{equation}  Part $1$ of theorem~\ref{theorem1A} applied to~\eqref{ewrkwwkkrewgm} implies: \begin{equation}
			Krull(A_{\ma}) -D(\kiki) -fd_{\ma}(A_{\ma})
			= fd_{A_{\ma}}(A_{\ma}) - D(\kiki)
			-fd_{\ma}(A_{\ma})
			 \leq HCdim(A|k).  
			\end{equation}
			Moreover the \textit{characterization of quasi-freeness} given in corollary~\ref{propQF} implies that $A$ cannot be quasi-free if: \begin{equation}
			2+D(\kiki) -fd_{\ma}(A_{\ma})\leq Krull(A_{\ma}).  
			\end{equation}
			\item[Case 2: $Krull(A_{\ma})$ is infinite] \hfill \\
			For every positive integer $d$ there exists an $\ama$-regular sequence $x_1^d,..,x^d_d$ in $\ma$ of length $d$.  Therefore proposition~\ref{propflatcalc} implies:
			\begin{equation}
			(\forall d \in \mathbb{Z}^{+}) \mbox{ } d=fd_{A_{\ma}}(A_{\ma}/(x_1^d,..,x^d_d)).   
			\label{ewrkwwkkrewg1m}
			\end{equation}  Therefore part one of theorem~\ref{theorem1A} implies: 
			\small
			\begin{equation}
			(\forall d \in \mathbb{Z}^{+}) \mbox{ } d-D(\kiki )
			-fd_{\ma}(A_{\ma})
			=fd_{A_{\ma}}(A_{\ma}/(x_1^d,..,x^d_d)) - D(\kiki )
			-fd_{\ma}(A_{\ma})
			\leq HCdim(A|k).   
			\label{ewrkwwkkrewg2m}
			\end{equation}
			\normalsize
			Since $D(k)$ and $fd_{\ma}(A_{\ma})$ are finite: \begin{equation}
			\infty -D(\kiki ) -fd_{\ma}(A_{\ma})= \infty \leq HCdim(A|k).  
			\label{nonononoooo124r325231}
			\end{equation}
			Since $Krull(A_{\ma})$ is infinite \eqref{nonononoooo124r325231}
			implies: \begin{equation}
			Krull(A_{\ma}) - D(\kiki ) -fd_{\ma}(A_{\ma})= \infty = HCdim(A|k).  
			\label{nonononoooo124r32523}
			\end{equation}
			In this case corollary~\ref{propQF} implies that $A$ is not quasi-free.  
		\end{description}
	\end{enumerate}
\end{proof}
\subsection{Proof of Theorem~\ref{thrm_wkD1Ctada}}
\begin{proof}[{Proof of Theorem~\ref{thrm_wkD1Ctada}}]  For any $A$-modules $M$ and $N$ lemma $\autoref{lem:pre-weakWeibeliso}$ implied:
		\begin{equation}
		Ext_{\ea}^{\star}(N,M) \cong HH^{\star}(A,Hom_k(N,M))
		\label{gfghr4eu9gh948h48gjg}.
		\end{equation}  Therefore taking supremums over all the $A$-modules $M,N$, of the integers $n$ for which $\eqref{gfghr4eu9gh948h48gjg}$ is non-trivial implies:
		\small\begin{equation}
		D_{{\mathscr{E}^k}}(A)= \underset{M,N\in _AMod}{sup} (sup(\{n \in \mathbb{N}^{\#} | Ext^n(M,N)\neq 0 \})) \end{equation}
		\begin{equation}
		= \underset{M,N\in _AMod}{sup} (sup(\{n \in \mathbb{N}^{\#} | HH^{n}(A,Hom_k(N,M))\neq 0 \}))
		\label{dgfohrg9hg943hg0934jg}.  
		\end{equation}\normalsize  $Hom_k(N,M)$ is only a particular case of an $A^e$-module; therefore taking supremums over \textit{all} $A$-modules bounds $\eqref{dgfohrg9hg943hg0934jg}$ above as follows:
		\small
		\begin{equation}
		D_{{\mathscr{E}^k}}(A)
		= \underset{M,N\in _AMod}{sup} (sup(\{n \in \mathbb{N}^{\#}| HH^{\star}(A,Hom_k(N,M))\neq 0 \})) \end{equation}
		\begin{equation}
		\leq 
		\underset{\tilde{M}\in _{A^e}Mod}{sup} (sup(\{n \in \mathbb{N}^{\#} | HH^{n}(A,\tilde{M})\neq 0 \})). \label{gfdgg54by65yb6e5ybe65}
		\end{equation}\normalsize  
		The right hand side of $\eqref{gfdgg54by65yb6e5ybe65}$ is precisely the definition of the Hochschild cohomological dimension.  Therefore \begin{equation}
		D_{\ek}(A)\leq HCdim(A|k)
		\label{eqeee1i1ii1i1i1i1ij2iu32u9h439iuhvu93i344445}
		\end{equation}
		Proposition~\ref{propprojtoGlb} applied to $\eqref{eqeee1i1ii1i1i1i1ij2iu32u9h439iuhvu93i344445}$ then draws out the conclusion.  

Proposition $\autoref{propflatcalc}$ implies that: \begin{equation}
		n=fd_A(A/(x_1,..,x_n)).  
		\label{memememememrfeifiefi4iji493}
		\end{equation}
		Therefore $\eqref{eq1dcerion1}$ applied to the $A$-module $A/(x_1,..,x_n$ together with $\eqref{memememememrfeifiefi4iji493}$ imply: \begin{equation}
		n -D(k) = fd_A(A/(x_1,..,x_n) \leq D_{\mathscr{E}^{k}} \leq HCDim(A/k).  
		\end{equation}
		
		If $\ncof$ is generated by a  regular sequence $x_1,..,x_n$ then proposition $\autoref{propflatcalc}$ implies: \begin{equation}
		n = fd_{A^e}(A\otimes_k A/\ncof)
		\end{equation}
		However by definition of $\ncof$ as the kernel of $\mu_A$: $A \otimes_k A/\ncof \cong A$.  Therefore: \begin{equation}
		n = fd_{A^e}(A).
		\end{equation}
		Lemma $\autoref{lem:fdpd}$ together with lemma~\ref{thrmHochschild1958} imply: \begin{equation}
		n = fd_{A^e}(A) \leq pd_{A^e}(A) \leq pd_{\rh}(A) + D(k).  
		\label{jfvo43ufbv43o8u8948484}
		\end{equation}
		Since $D(k)$ is finite then $\eqref{jfvo43ufbv43o8u8948484}$ entails: \begin{equation}
		n -D(k) \leq pd_{\rh}(A).
		\label{4fb54tb5n5n5}
		\end{equation}
		By theorem $\autoref{thrmcharnsmooth}$ $\eqref{4fb54tb5n5n5}$ is equivalent to: \begin{equation}
		n -D(k) \leq HCDim(A).
		\end{equation}
		
		If $A$ is Cohen-Macaulay at one of its maximal ideals $\ma$ then there exists a maximal regular $x_1,..,x_d$ in $A_{\ma}$ with $d=Krull(A_{\ma})$.   Therefore $\eqref{eq2dcerion2}$ implies: \begin{equation}
		Krull(A_{\ma}) -D(k) = d -D(k) \leq D(A_{\ma}) - D(k).
		\label{fefefefefefe333}
		\end{equation}
		Since $D(A_{\ma})\leq D(A)$, then
		\begin{equation}
		Krull(A_{\ma}) -D(k) \leq D(A_{\ma}) - D(k) \leq D(A) - D(k). 
		\label{fg4f4v43v44v4b}
		\end{equation}
		Finally $\eqref{eq1dcerion1}$ applied to $\eqref{fg4f4v43v44v4b}$ implies: \begin{equation}
		Krull(A_{\ma}) -D(k) \leq D(A) - D(k) \leq HCDim(A). 
		\end{equation}
\end{proof}
Next, we summarize the contributions made within this paper.  
\section{Conclusion}
In this paper, we extended a result of \cite{cuntz1995algebra} which showed that most commutative affine $k$-algebras fail to be smooth in the non-commutative sense, as formalized by quasi-freeness.   This was possible by using Theorem~\ref{thrm_wkD1Ctada} which established a concrete lower-bound on the Hochschild cohomological dimension of a commutative $k$-algebra in terms of a small number of classical dimension-theoretic invariants.  In particular, a simple computation only involving the Krull dimension of $A$, the flat-dimension of $k$  at one point, and the global dimension of the base ring, can be used to determine if a non-commutative space's associated $k$-algebra is quasi-free or not.  

The author would like to thank ETH Z\"{u}rich foundation for its support as well as the Universit\'{e} de Montr\'{e}al's Mathematics department.  

%
%\begin{appendices}
%
%
%\end{appendices}	

	\renewcommand{\thepage}{}
\bibliographystyle{abbrvnat}
\bibliography{References}
\begin{appendices}
	This appendix contains certain technical lemmas or auxiliary results that otherwise detracted from the overall flow of the paper.  
\section{Some Technical Results}
\begin{prop}[Dimension Shifting]\label{propext11}	
	If
	\begin{equation}
	\begin{tikzpicture}[>=angle 90]
	\matrix(a)[matrix of math nodes,
	row sep=3em, column sep=2.5em,
	text height=1.5ex, text depth=0.25ex]
	{...\overset{d_{n+1}}{\rightarrow} P_n & ... & P_1 & P_0 & 0 \\};
	\path[->,font=\scriptsize]
	(a-1-1) edge node[above]{$ d_{n} $} (a-1-2);
	\path[->,font=\scriptsize]
	(a-1-2) edge node[above]{$ d_2 $} (a-1-3);
	\path[->,font=\scriptsize]
	(a-1-3) edge node[above]{$ d_1 $} (a-1-4);
	\path[->,font=\scriptsize]
	(a-1-4) edge node[above]{$  $} (a-1-5);
	\end{tikzpicture} 
	\label{f4f4f4f4tseq1}
	\end{equation}
	is a deleted $\ea$-projective resolution of an $A$-module $M$ then for every $A$-module $N$ and for every positive integer $n$ there are isomorphisms natural in $N$: \begin{equation}
	Ext^1_{\ea}(Ker(d_n),N)\cong Ext^{n+1}_{\ea}(A,N)
	\end{equation}
\end{prop}
\begin{proof}
	By definition the truncated sequence is exact: \begin{equation}
	\begin{tikzpicture}[>=angle 90]
	\matrix(a)[matrix of math nodes,
	row sep=3em, column sep=2.5em,
	text height=1.5ex, text depth=0.25ex]
	{...\overset{d_{n+j}}{\rightarrow} P_{n+j} & ... & P_{n+1} & Ker(d_n) & 0 \\};
	\path[->,font=\scriptsize]
	(a-1-1) edge node[above]{$ d_{n+j-1} $} (a-1-2);
	\path[->,font=\scriptsize]
	(a-1-2) edge node[above]{$ d_{n+1} $} (a-1-3);
	\path[->,font=\scriptsize]
	(a-1-3) edge node[above]{$ \eta $} (a-1-4);
	\path[->,font=\scriptsize]
	(a-1-4) edge node[above]{$  $} (a-1-5);
	\end{tikzpicture}
	\label{f4f4f4f4tseq2},
	\end{equation}
	where $\eta$ is the canonical map satisfying $d_n=ker(d_n)\circ \eta$ (arising from the universal property of $ker(d_n)$).  Moreover since~\eqref{f4f4f4f4tseq1} is $\ea$-exact, $d_n$ is $k$-split; whence $\eta$ must be $k$-split.  Moreover for every $j\geq n+1$, $d_j$ was by assumption $k$-split therefore~\eqref{f4f4f4f4tseq2} is $\ea$-exact and since for every natural number $m>n$ $P_m$ is by hypothesis $\ea$-projective then~\eqref{f4f4f4f4tseq2} is an augmented $\ea$-projective resolution of the $A$-module $Ker(d_n)$.  
	
	For every natural number $m$, relabel: \begin{equation}
	Q_m:=P_{m+n} \mbox{ and } p_{m}:=d_{n+m}.  
	\end{equation}
	By theorem~\ref{thmcomparisonrelcohomolg}: \begin{equation}
	(\forall N \in _AMod)(\forall m \in \mathbb{N}) \mbox{ } 
	Ext_{\ea}^m(Ker(d_n),N)
	\cong
	H^{m}(Hom_A(Q_{\star},N))
	\end{equation}
	\begin{equation}
	= Ker(Hom_A(p_n,N))/Im(Hom_A(p_{n+1},N))
	\end{equation}
	\begin{equation}
	=
	Ker(Hom_A(d_{n+m},N))/Im(Hom_A(d_{n+m+1},N))
	\end{equation}
	\begin{equation}
	=H^{m+n}(Hom_A(P_{\star},N)) 
	\end{equation}
	\begin{equation}
	\cong
	Ext_{\ea}^m(A,N).  
	\end{equation}
\end{proof}
\section{Auxiliary Results}
\begin{proof}[Proof of Proposition~\ref{lemqfromfqfandprojrel1}]
	Let \begin{equation}
	0 \rightarrow M \rightarrow B \overset{\pi}{\rightarrow} T_A(P) \rightarrow 0
	\label{vrvr45r4v4e45v5v5}
	\end{equation}
	be a $k$-Hochschild extension of $T_A(P)$ by $M$.  
	We use the universal property of $T_A(P)$ to show that there must exist a lift $l$ of~\eqref{vrvr45r4v4e45v5v5}.  
	
	Let $p: T_A(P) \rightarrow A$ be the projection $k$-algebra homomorphism of  $T_A(P)$ onto $A$.  $p$ is $k$-split since the $k$-module inclusion $i: A \rightarrow T_A(P)$ is a section of $p$; therefore $p$ is an $\rh$-epimorphism and \begin{equation}
	0 \rightarrow Ker(p \circ \pi ) \rightarrow B \rightarrow A \rightarrow 0
	\end{equation}
	is a $k$-Hochschild extension of $A$ by the $(A,A)$-bimodule $Ker(p \circ \pi )$.  Since $A$ is a quasi-free $k$-algebra there exists a $k$-algebra homomorphism $l_1: A \rightarrow B$ lifting $p \circ \pi$.  Hence $B$ inherits the structure of an $(A,A)$-bimodule and $\pi$ may be viewed as an $(A,A)$-bimodule homomorphism.  Moreover $l_1$ induces an $A$-algebra structure on $B$.  
	
	Let $f: P \rightarrow T_A(P)$ be the $(A,A)$-bimodule homomorphism satisfying the universal property of the tensor algebra on the $(A,A)$-bimodule $P$.  Since $\pi: B \rightarrow A$ is an $\rh$-epimorphism 
	%(since~\eqref{vrvr45r4v4e45v5v5} was a $k$-Hochschild extension) 
	and since $P$ is an $\rh$-projective $(A,A)$-bimodule, proposition~\ref{eaprojchar} implies that that there exists an $(A,A)$-bimodule homomorphism $l_2 : P \rightarrow B$ satisfying $\pi \circ l_2 = f$.  
	
	Since $l_2:P \rightarrow B$ is an $(A,A)$-bimodule homomorphism to a $A$-algebra the universal property of the tensor algebra $T_A(P)$ on the $(A,A)$-bimodule $P$, see \cite{bourbaki1998algebra}, implies there is an $A$-algebra homomorphism $l: T_A(P) \rightarrow B$ whose underlying function satisfies: $l \circ f = l_2$.  
	
	Therefore $l \circ \pi \circ l_2 = l_2$; whence $l \circ \pi = 1_{T_A(P)}$; that is $l$ is a $A$-algebra homomorphism which is a section of $\pi$, that is $l$ lifts $\pi$.  
\end{proof}
\end{appendices}
\printindex
\end{document}